\documentclass[reqno]{amsart}

\usepackage{amsmath,amsfonts,amssymb}
\usepackage{times}
\usepackage[centertags]{amsmath}
\usepackage{amsfonts}
\usepackage{amssymb}
\usepackage{amsthm}
\usepackage{newlfont}

\newtheorem{thm}{Theorem}

\newtheorem{lem}{Lemma}
\newtheorem{remark}{Remark}
\newtheorem{definition}{Definition}

\begin{document}

\begin{center}
{\Large\bf New Inequalities of the type of Hadamard's through $s-(\alpha,m)$ Co-ordinated Convex functions}
\vspace*{0.5cm}

{M. I. Bhatti}$^{(1)}$, {M. Muddassar}$^{(2\,*)}$ and {F. Yasin}$^{(1)}$

%%%%%%%%%%%%%%%%%%%%%%%%%%%%%%%%%%%%%%%%%%%%%%%%%%%%%%%%%%
%
% The name of the author who will deliver a talk is underlined
%
%%%%%%%%%%%%%%%%%%%%%%%%%%%%%%%%%%%%%%%%%%%%%%%%%%%%%%%%%%

\end{center}

\vspace*{-0.3cm}

\begin{center}
$^{(1)}$ \small{\textit{Department of Mathematics, University of Engineering \& Technology, Lahore, Pakistan, uetzone@hotmail.com, farkhanda.yasin@yahoo.com}}\\
$^{(2)}$ \small{\textit{Department of Mathematics, University of Engineering \& Technology, Taxila, Pakistan, malik.muddassar@gmail.com ($^{(*)}$ \small{\textit{for correspondence}})}}\\
\end{center}

\date{\today}
\subjclass[2010]{26D15, 26A51, 39A10.}
\keywords{Hadamard type inequality, $s-(\alpha,m)$-Convex function, Co-ordinated Convex functions, H\"{o}lder's Integral Inequality, Power mean Inequality.}

\begin{center}
{\bf Abstract:} This monograph is associated with the renowned Hermite-Hadamard's integral inequality of $2$-variables on the co-ordinates. In this article we established several inequalities of the type of Hadamard's for the mappings whose absolute values of second order partial derivatives are $s-(\alpha,m)$-convex mappings.
\end{center}

{\setcounter{section}{0}}
\markboth{\underline{\hspace{3.45in}M. Muddassar, M. I. Bhatti and F. Yasin}}
{\underline{\hspace{1pt}Inequalities through Generalized Co-ordinated Convex functions\hspace{2.60in}}}\pagestyle{myheadings}

\section{Introduction}\label{sec1}
The role of inequalities involving convex functions is greeted from the very beginning and is now widely acknowledged as one of the prime dividing forces behind the evolution of several modern areas of mathematics and has been given considerable attention. One of the well-nigh notable inequality for convex mappings is declared as Hermite-Hadamard(Hadamard) inequality. Define as\\
Let a function $f:I \subset \mathbb{R} \rightarrow \mathbb{R}$ is convex on $I$, induces
\begin{equation}\label{HH}
    f\left(\frac{a+b}{2}\right)\leq\frac{1}{b-a}\int_a^b f(x) dx\leq\frac{f(a)+f(b)}{2}
\end{equation}
for $a,b\in I,$ with $a<b$.\\
This inequality gives us an estimate, from below and from above, of the average value of $f : [a, b] \rightarrow\mathbb{R}$ for $a,b\in I,$ with $a<b$.\\
Since 1893 when Hadamard showed his renowned inequality, many mathematicians have been working close to it, and generalizing and refining this inequality in many dissimilar ways and with a lot of applications. In this article, we are going to talked about some inequalities linked with Hermite-Hadamard's on co-ordinates.\\
Consider a bidimensional interval $\Delta =: [l, m] \times [p, q]$ in $\mathbb{R}^2$ with $l < m$ and $p < q$. A map $f : \Delta \rightarrow \mathbb{R}$ is supposed to be convex on $\Delta$ when the coming inequality:
$$f(\alpha u + (1 - \alpha) v, \alpha w + (1 - \alpha)z) \leq  \alpha f (u, w) + (1 - \alpha) f (v,z)$$
agrees, $\forall\,\, (u, w) , (v,z) \in \Delta\,\,\wedge\,\,\alpha \in [0, 1]$. A map $f : \Delta \rightarrow \mathbb{R}$ is supposed to be convex on the co-ordinates $(\Delta)$ when the partial mappings $f_y : [l, m] \rightarrow \mathbb{R}$, $f_y (\zeta) = f (\zeta, y)$ and $f_x : [p, q] \rightarrow \mathbb{R}$, $f_x (\eta) = f (x, \eta)$ are convex, $\forall\,\,x \in [l, m]\,\,\wedge\,\,y \in [p, q]$ (go through \cite{r2}).
A schematic explanation for co-ordinated convex functions might be posited as:
\begin{definition}\label{d1}
A function $f : \Delta \rightarrow \mathbb{R}$ supposed to ba a co-ordinated convex on $\Delta\,\,\forall\,\, \alpha, \beta \in [0, 1]$ and $(\mu, \nu), (\phi, \psi) \in \Delta$, when the coming inequality agrees:
\begin{eqnarray*}
&&\!\!\!\!\!\!\!\!\!\!f(\alpha \mu + (1 - \alpha) \nu, \beta \phi + (1 - \beta) \psi) \leq\\
&&\alpha \beta f(\mu, \phi) + \beta(1 - \alpha)f(\nu, \phi) + \alpha(1 - \beta)f(\mu, \psi) + (1 - \alpha)(1 - \beta)f(\nu, \psi).
\end{eqnarray*}
\end{definition}
Clearly, all convex functions defined on $\mathbb{R}$ is also convex functions on co-ordinates $\mathbb{R}^2$. In spite of this, there are co-ordinated convex functions which are not convex in $\mathbb{R}$ \cite{r2}. Likewise, in \cite{r2}, S. Dragomir found the next inequality of Hadamard's type for co-ordinated convex functions on a rectangular plane $\mathbb{R}^2$. See bibliography for various latest consequences referring Hermite-Hadamard's inequality for convex function on co-ordinates $\mathbb{R}^2$.
\begin{thm}\label{t1}
Let a function $f : \Delta \rightarrow \mathbb{R}$ is co-ordinated convex on $\Delta$, induces
\begin{eqnarray*}
&&\!\!\!\!\!\!\!\!\!\!\!\!\! f\left(\frac{a+b}{2}, \frac{c+d}{2}\right) \leq \frac{1}{2}\left[\frac{1}{b-a}\int_a^b f\left(u, \frac{c+d}{2}\right)du + \frac{1}{d-c}\int_c^d f\left(\frac{a+b}{2}, v\right)dv\right]\\
&&\indent\indent\indent\indent\indent \leq \frac{1}{(b-a)(d-c)} \int_a^b \int_c^d f(u,v)dvdu \\
&&\indent\indent\indent\indent\indent \leq \frac{1}{4}\left[\frac{1}{b-a}\int_a^b f(u, c)du + \frac{1}{b-a}\int_a^b f(u, d)du\right.\\
 && \indent\indent\indent\indent\indent\indent\indent\indent\indent\indent \left.+\frac{1}{d-c} \int_c^d f(a, v)dv + \frac{1}{d-c} \int_c^d f(b, v)dv\right]\\
&&\indent\indent\indent\indent\indent \leq \frac{f(a, c)+f(a,d)+f(b,c) + f(b,d)}{4}
\end{eqnarray*}
\end{thm}
M. Z. Sarikaya et. al. in \cite{r7} proved an identity as below and build few more Hermite-Hadamard's type integral inequalities for co-ordinated convex functions on rectangular plane.
\begin{lem}\label{L1}
Let $f : \Delta \subset \mathbb{R}^2 \rightarrow \mathbb{R}^2$ be a partially differentiable function on $\Delta :=
[a, b] \times [c, d]$ in $\mathbb{R}^2$ with $a < b\,\,\wedge\,\,c < d$. If $\frac{\partial ^2f}{\partial \lambda \partial \mu} \in L(\Delta)$, for $\lambda, \mu \in (0, 1]$ then the following equality agrees:
\begin{eqnarray}\label{le1}
&&\nonumber\!\!\!\!\!\!\!\!\!\! \frac{f(a, c)+f(a,d)+f(b,c) + f(b,d)}{4} + \frac{1}{(b-a)(d-c)} \int_a^b \int_c^d f(x,y)dydx\\
&&\nonumber\indent\indent\indent - \frac{1}{2}\left[\frac{1}{b-a}\int_a^b [f(x,c)+f(x,d)]dx + \frac{1}{d-c}\int_c^d [f(a,y)+f(b,y)]dy\right]\\
&&=\!\!\frac{(b\!-\!a)\!(d\!-\!c)}{4}\!\!\int_0^1\!\!\!\int_0^1\!\!\! (1\!-\!2\lambda)(1\!-\!2\mu)\!\frac{\partial ^2f}{\partial \lambda \partial \mu}\! \left(\lambda a\!+\!(1-\lambda)b, \mu c\!+\!(1-\mu)d\right)d\lambda d\mu
\end{eqnarray}
\end{lem}
He demonstrate in this way:
\begin{thm}\label{t2}
Let $f : \Delta \subset \mathbb{R}^2 \rightarrow \mathbb{R}^2$ be a partially differentiable function on $\Delta :=[a, b] \times [c, d]$ in $\mathbb{R}^2$ with $a < b\,\,\wedge\,\,c < d$. If $\left| \frac{\partial ^2f}{\partial t \partial s}\right|$ is a convex function on $\Delta$, induces
\begin{eqnarray}\label{te2}
&&\nonumber\!\!\!\!\!\!\!\!\!\! \left|\frac{f(a, c)+f(a,d)+f(b,c) + f(b,d)}{4} + \frac{1}{(b-a)(d-c)} \int_a^b \int_c^d f(x,y)dydx - A\right|\\
&&\!\!\!\! \leq\!\! \frac{(b-a)(d-c)}{16}\!\!\left(\!\frac{\left| \frac{\partial ^2f}{\partial t \partial s}(a, c)\right|+\left| \frac{\partial ^2f}{\partial t \partial s}(a, d)\right|+\left| \frac{\partial ^2f}{\partial t \partial s}(b, c)\right|+\left| \frac{\partial ^2f}{\partial t \partial s}(b, d)\right|}{4}\!\right)
\end{eqnarray}
where $$\mathcal{A}= \frac{1}{2}\left[\frac{1}{b-a}\int_a^b [f(x,c)+f(x,d)]dx + \frac{1}{d-c}\int_c^d [f(a,y)+f(b,y)]dy\right]$$
\end{thm}
Two more results were presented in same related to absolute value of $q^{th}$-derivative of a map is convex on co-ordinates on $\Delta$. This paper is in the continuation of \cite{r7}, which provide more general and refine results as presented in \cite{r7}.\\
    The arrangement of this article as:, After this introduction, in section \ref{Sec2} we specify a generalized convex functions and talked about few more inequalities for generalized co-ordinated convex functions linked with Hadamard's inequality.
\section{principal Outcomes}\label{Sec2}
To obtain our major consequences, we first incur the definitions below.
\begin{definition}\label{d2}
A function $f : \Delta \subset \mathbb{R}^2 \rightarrow \mathbb{R}^2$ is supposed to be co-ordinated $s-(\alpha,m)$-convex function on $\Delta$ in first sense or $f$ belonging to the class ${K_{m, 1}^{\alpha, s}}(\Delta)$ , if $\forall\,\,x, y, z, w \in \Delta\, \wedge\, \lambda, \mu \in  [0, 1]$, the coming inequality agrees:
\begin{eqnarray*}
&&\!\!\!\!\!\!\!\!\!\!\!\! f(\lambda x \!+\!(1\!-\!\lambda)z, \mu y \!+\! (1\!-\!\mu) w) \leq \lambda^{\alpha_1 s_1}\mu^{\alpha_2 s_2}f(x,y)\!+\!m_2\lambda^{\alpha_1 s_1}(1\!-\!\mu^{\alpha_2 s_2})f\left(x, \frac{w}{m_2}\right)\\
&&\indent +\!m_1\mu^{\alpha_2 s_2}(1\!-\!\lambda^{\alpha_1 s_1})f\left(\frac{z}{m_1},y\right)\!+\!m_1m_2(1\!-\!\lambda^{\alpha_1 s_1})(1\!-\!\mu^{\alpha_2 s_2})f\left(\frac{z}{m_1},\frac{w}{m_2} \!\right)
\end{eqnarray*}
If the partially defined mappings $f_x(x,v)$ and $f_y(u,y)$ are $s-(\alpha, m)$ convex for some $s_1, s_2 \in (0, 1]$ and $(\alpha_1,m_1), (\alpha_2,m_2) \in [0,1]^2$.
\end{definition}
\begin{definition}\label{d3}
A function $f : \Delta \subset \mathbb{R}^2 \rightarrow \mathbb{R}^2$ is supposed to be co-ordinated $s-(\alpha,m)$-convex function on $\Delta$ in first sense or $f$ belonging to the class ${K_{m, 2}^{\alpha, s}}(\Delta)$ , if $\forall\,\,x, y, z, w \in \Delta\,\,\wedge\,\,\lambda, \mu \in  [0, 1]$, the coming inequality agrees:
\begin{eqnarray*}
&&\!\!\!\!\!\!\!\!\!\!\!\! f(\lambda x \!+\!(1\!-\!\lambda)z, \mu y \!+\! (1\!-\!\mu) w) \leq \lambda^{\alpha_1 s_1}\mu^{\alpha_2 s_2}f(x,y)\!+\!m_2\lambda^{\alpha_1 s_1}(1\!-\!\mu^{\alpha_2})^{s_2}f\left(x, \frac{w}{m_2}\right)\\
&&\indent +\!m_1\mu^{\alpha_2s_2}(1\!-\!\lambda^{\alpha_1})^{s_1}f\left(\frac{z}{m_1},y\right)\!+\!m_1m_2(1\!-\!\lambda^{\alpha_1})^{s_1}(1-\mu^{\alpha_2})^{s_2}f\left(\frac{z}{m_1},\frac{w}{m_2} \right)
\end{eqnarray*}
If the partially defined mappings $f_x(x,v)$ and $f_y(u,y)$ are $s-(\alpha, m)$ convex for some $s_1, s_2 \in (0, 1]$ and $(\alpha_1,m_1), (\alpha_2,m_2) \in [0,1]^2$.
\end{definition}
\begin{thm}\label{T1}
 Let  a function $f : \Delta \subset \mathbb{R}^2 \rightarrow \mathbb{R}^2$ be a partially differentiable on $\Delta := [a, b] \times [c, d]$ in $\mathbb{R}^2$ with $a < b\,\,\wedge\,\,c < d$. If $\left| \frac{\partial ^2f}{\partial t \partial s}\right|$ belonging to the class ${K_{m, 1}^{\alpha, s}}(\Delta)$ and $\frac{\partial ^2f}{\partial \lambda \partial \mu} \in L(\Delta)$, induces the following inequality agrees:
 \begin{eqnarray}\label{te1}
&&\nonumber\!\!\!\!\!\!\!\!\!\!\!\!\left|\frac{f(a, c)+f(a,d)+f(b,c) + f(b,d)}{4} + \frac{1}{(b-a)(d-c)} \int_a^b \int_c^d f(x,y)dydx - \mathcal{A}\right|\\
&&\nonumber\!\!\!\!\!\! \leq \frac{(b-a)(d-c)}{4}\left[\mathcal{B}\mathcal{C}\left\{\left|\frac{\partial ^2f}{\partial \lambda \partial \mu}(a,c)\right|+m_1\left|\frac{\partial ^2f}{\partial \lambda \partial \mu}\left(\frac{b}{m_1}, c\right)\right|\right\}\right.\\
&&\!\!\!\!\!\!\left.+\!\left(\!\frac{1}{2}\!-\!\mathcal{B}\!\right)\left(\!\frac{1}{2}\!-\!\mathcal{C}\!\right)\left\{m_2\left|\frac{\partial ^2f}{\partial \lambda \partial \mu}\left(a, \frac{d}{m_2}\right)\right|+m_1m_2\left|\frac{\partial ^2f}{\partial \lambda \partial \mu}\left(\frac{b}{m_1}, \frac{d}{m_2}\right)\right|\right\}\right]
\end{eqnarray}
where $$\mathcal{B}=\frac{1}{2^{\alpha_1s_1}(\alpha_1s_1+1)}-\frac{1}{2^{\alpha_1s_1}(\alpha_1s_1+2)}+\frac{2}{\alpha_1s_1+2}-\frac{1}{\alpha_1s_1+1}$$
and
$$\mathcal{C}=\frac{1}{2^{\alpha_2s_2}(\alpha_2s_2+1)}-\frac{1}{2^{\alpha_2s_2}(\alpha_2s_2+2)}+\frac{2}{\alpha_2s_2+2}-\frac{1}{\alpha_2s_2+1}$$
\end{thm}
\begin{proof}
Taking numerical value of equation \ref{le1}, implies
\begin{eqnarray}\label{t1a}
&&\nonumber\!\!\!\!\!\!\!\!\!\! \left|\frac{f(a, c)+f(a,d)+f(b,c) + f(b,d)}{4} + \frac{1}{(b-a)(d-c)} \int_a^b \int_c^d f(x,y)dydx \!-\!\mathcal{A}\right|\\
&&\!\!\!\!\!\!\! \leq \!\!\frac{(\!b\!-\!a\!)\!(\!d\!-\!c\!)}{4}\!\!\int_0^1\!\!\!\! \int_0^1\!\!\! |(\!1\!-\!2\lambda\!)||(\!1\!-\!2\mu\!)| \left|\!\frac{\partial ^2f}{\partial \lambda \partial \mu} \left(\lambda a\!+\!(1\!-\!\lambda)b, \mu c\!+\!(1\!-\!\mu)d\right)\!\right|d\lambda d\mu
\end{eqnarray}
Since $\left| \frac{\partial ^2f}{\partial t \partial s}\right|$ is $s-(\alpha,m)$ co-ordinated convex on $\Delta$, then we have
\begin{eqnarray}\label{t1b}
&&\nonumber\!\!\!\!\!\!\!\!\!\! \int_0^1\!\! \int_0^1 |(1-2\lambda)||(1-2\mu)| \left|\frac{\partial ^2f}{\partial \lambda \partial \mu} \left(\lambda a\!+\!(1-\lambda)b, \mu c\!+\!(1-\mu)d\right)\right|dtds\\
&&\nonumber\indent \leq \int_0^1\!\! \int_0^1 |(1-2\lambda)||(1-2\mu)|\left[\lambda^{\alpha_1 s_1}\mu^{\alpha_2 s_2}\left|\frac{\partial ^2f}{\partial \lambda \partial \mu}(a, c)\right|\right.\\
&&\nonumber\indent\left.+\!m_2\!\lambda^{\alpha_1 s_1}\!(\!1\!-\!\mu^{\alpha_2 s_2}\!)\!\left|\frac{\partial ^2f}{\partial \lambda \partial \mu}\!\left(\!a, \frac{d}{m_2}\!\right)\right|\!+\!m_1\!\mu^{\alpha_2 s_2}\!(\!1\!-\!\lambda^{\alpha_1 s_1}\!)\left|\frac{\partial ^2f}{\partial \lambda \partial \mu}\left(\!\frac{b}{m_1}, c\!\right)\right|\right. \\
&&\indent\indent\indent\indent\indent\indent\indent\left. +m_1m_2(1-\lambda^{\alpha_1 s_1})(1-\mu^{\alpha_2 s_2})\left|\frac{\partial ^2f}{\partial \lambda \partial \mu}\left(\frac{b}{m_1}, \frac{d}{m_2}\right)\right|\right]
\end{eqnarray}
Now consider the integral
\begin{eqnarray}\label{t1c}
&&\nonumber\!\!\!\!\!\!\!\!\!\!\!\!\!\!\!\! \int_0^1\!\!|(1-2\mu)|\!\!\left[\!\! \int_0^{\frac{1}{2}}\!\! |(1\!-\!2\lambda)|\! \left\{\lambda^{\alpha_1 s_1}\mu^{\alpha_2 s_2}\left|\frac{\partial ^2f}{\partial \lambda \partial \mu}(a, c)\right|\!+\!m_2\!\lambda^{\alpha_1 s_1}\!(\!1\!-\!\mu^{\alpha_2 s_2}\!)\!\left|\frac{\partial ^2f}{\partial \lambda \partial \mu}\!\left(\!a, \frac{d}{m_2}\!\right)\right|\right.\right.\\
&&\nonumber\!\!\!\!\!\!\!\!\!\!\!\!\!\!\!\!\!\!\left.\left.+\!m_1\!\mu^{\alpha_2 s_2}\!(\!1\!-\!\lambda^{\alpha_1 s_1}\!)\left|\!\frac{\partial ^2f}{\partial \lambda \partial \mu}\left(\!\frac{b}{m_1},\! c\!\right)\!\right|\!+\!m_1m_2(1\!-\!\lambda^{\alpha_1 s_1})(\!1\!-\!\mu^{\alpha_2 s_2}\!)\left|\!\frac{\partial ^2f}{\partial \lambda \partial \mu}\left(\!\frac{b}{m_1},\! \frac{d}{m_2}\!\right)\!\right|\!\right\}\!d\lambda\right. \\
&&\nonumber\!\!\!\!\!\!\!\!\!\!\!\!\!\!\!\!\!\!\!\!\left. +\!\!\int_{\frac{1}{2}}^1 |(1-2\lambda)| \left\{\lambda^{\alpha_1 s_1}\mu^{\alpha_2 s_2}\left|\frac{\partial ^2f}{\partial \lambda \partial \mu}(a, c)\right|+\!m_2\!\lambda^{\alpha_1 s_1}\!(\!1\!-\!\mu^{\alpha_2 s_2}\!)\!\left|\frac{\partial ^2f}{\partial \lambda \partial \mu}\!\left(\!a, \frac{d}{m_2}\!\right)\right|\right.\right.\\
&&\!\!\!\!\!\!\!\!\!\!\!\!\!\!\!\!\!\!\!\left.\left.\!+\!m_1\!\mu^{\alpha_2 s_2}\!(\!1\!\!-\!\!\lambda^{\alpha_1 s_1}\!)\left|\!\frac{\partial ^2f}{\partial \lambda \partial \mu}\!\left(\!\!\frac{b}{m_1}\!,\! c\!\!\right)\!\right|\!+\!m_1m_2(\!1\!-\!\lambda^{\alpha_1 s_1}\!)(\!1\!-\!\mu^{\alpha_2 s_2}\!)\left|\!\!\frac{\partial ^2f}{\partial \lambda \partial \mu}\!\left(\!\!\frac{b}{m_1}\!,\! \frac{d}{m_2}\!\!\right)\!\right|\!\right\}d\lambda\!\right]\!d\mu
\end{eqnarray}
Applying simple integration rules and simplifying, we have
\begin{eqnarray}\label{t1d}
&&\nonumber\!\!\!\!\!\!\!\!\!\!\!\!\!\!\!\! \int_0^1 |(1 -2\mu)|\left[\left\{\mu^{\alpha_2s_2}\left|\frac{\partial ^2f}{\partial \lambda \partial \mu}(a,c)\right|+m_2\left(1-\mu^{\alpha_2s_2}\right)\left|\frac{\partial ^2f}{\partial \lambda \partial \mu}\left(a, \frac{d}{m_2}\right)\right|\right\}\right.\\
&&\nonumber\indent\indent\indent\left.\left(\frac{1}{2^{\alpha_1s_1}(\alpha_1s_1+1)}-\frac{1}{2^{\alpha_1s_1}(\alpha_1s_1+2)}+\frac{2}{\alpha_1s_1+2}-\frac{1}{\alpha_1s_1+1}\right)\right.\\
&&\nonumber\indent\left.+\left\{m_1\mu^{\alpha_2s_2}\left|\frac{\partial ^2f}{\partial \lambda \partial \mu}\left(\frac{b}{m_1},c\right)\right|+m_1m_2\left(1-\mu^{\alpha_2s_2}\right)\left|\frac{\partial ^2f}{\partial \lambda \partial \mu}\left(\frac{b}{m_1}, \frac{d}{m_2}\right)\right|\right\}\right.\\
&&\indent\indent\left.\left(\!\frac{1}{2}\!-\!\frac{1}{2^{\alpha_1s_1}(\alpha_1s_1\!+\!1)}-\frac{1}{2^{\alpha_1s_1}(\alpha_1s_1\!+\!2)}+\frac{2}{\alpha_1s_1\!+\!2}-\frac{1}{\alpha_1s_1\!+\!1}\!\right)\!\right]d\mu
\end{eqnarray}
Now further, we have
\begin{eqnarray}\label{t1e}
&&\nonumber\!\!\!\!\!\!\!\!\!\!\!\!\!\!\!\! \int_0^{\frac{1}{2}} |(1 -2\mu)|\left[\left\{\mu^{\alpha_2s_2}\left|\frac{\partial ^2f}{\partial \lambda \partial \mu}(a,c)\right|+m_2\left(1-\mu^{\alpha_2s_2}\right)\left|\frac{\partial ^2f}{\partial \lambda \partial \mu}\left(a, \frac{d}{m_2}\right)\right|\right\}\left(\mathcal{B}\right)\right.\\
&&\nonumber\!\!\!\!\!\!\!\!\!\!\!\!\!\!\!\!\!\!\left.+\!\left\{\!m_1\mu^{\alpha_2s_2}\left|\frac{\partial ^2f}{\partial \lambda \partial \mu}\left(\frac{b}{m_1}\!,\!c\right)\right|\!\!+\!\!m_1m_2\left(1\!-\!\mu^{\alpha_2s_2}\right)\left|\frac{\partial ^2f}{\partial \lambda \partial \mu}\left(\frac{b}{m_1}\!,\! \frac{d}{m_2}\right)\right|\right\}\left(\!\frac{1}{2}\!-\!\mathcal{B}\!\right)\!\right]d\mu\\
&&\nonumber\!\!\!\!\!\!\!\!\!\!\!\!\!\!\!\!+\int_{\frac{1}{2}}^1 |(2\mu-1)|\left[\left\{\mu^{\alpha_2s_2}\left|\frac{\partial ^2f}{\partial \lambda \partial \mu}(a,c)\right|+m_2\left(1-\mu^{\alpha_2s_2}\right)\left|\frac{\partial ^2f}{\partial \lambda \partial \mu}\left(a, \frac{d}{m_2}\right)\right|\right\}\left(\mathcal{B}\right)\right.\\
&&\!\!\!\!\!\!\!\!\!\!\!\!\!\!\!\!\!\!\left.+\!\left\{\!m_1\mu^{\alpha_2s_2}\left|\frac{\partial ^2f}{\partial \lambda \partial \mu}\!\left(\frac{b}{m_1}\!,\!c\right)\!\right|\!\!+\!\!m_1m_2\left(1\!-\!\mu^{\alpha_2s_2}\right)\left|\frac{\partial ^2f}{\partial \lambda \partial \mu}\!\left(\frac{b}{m_1}\!,\! \frac{d}{m_2}\!\right)\right|\right\}\left(\!\frac{1}{2}\!-\!\mathcal{B}\!\right)\!\right]d\mu
\end{eqnarray}
After doing simple integration and appointing $$\mathcal{C}=\frac{1}{2^{\alpha_2s_2}(\alpha_2s_2+1)}-\frac{1}{2^{\alpha_2s_2}(\alpha_2s_2+2)}+\frac{2}{\alpha_2s_2+2}-\frac{1}{\alpha_2s_2+1}$$
we have (\ref{te1}). Which fills out the proof.
\end{proof}
\begin{remark}
By adjusting $s_1=s_2=1$ and $(\alpha_1, m_1)=(\alpha_2, m_2) = (1,1)$ inequality $(\ref{te1})$ cuts down to \cite[Theorem 2]{r7}.
\end{remark}
\begin{thm}\label{T2}
 Let  a function $f : \Delta \subset \mathbb{R}^2 \rightarrow \mathbb{R}^2$ be a partially differentiable on $\Delta := [a, b] \times [c, d]$ in $\mathbb{R}^2$ with $a < b\,\,\wedge\,\,c < d$. If $\left| \frac{\partial ^2f}{\partial t \partial s}\right|^q$, for $q>1$ belonging to the class ${K_{m, 1}^{\alpha, s}}(\Delta)$ and $\frac{\partial ^2f}{\partial \lambda \partial \mu} \in L(\Delta)$, then the following inequality agrees:
 \begin{eqnarray}\label{te2}
&&\nonumber\!\!\!\!\!\!\!\!\!\!\!\!\left|\frac{f(a, c)+f(a,d)+f(b,c) + f(b,d)}{4} + \frac{1}{(b-a)(d-c)} \int_a^b \int_c^d f(x,y)dydx - \mathcal{A}\right|\\
&&\nonumber\!\!\!\!\!\! \leq \frac{(b-a)(d-c)}{4(p+1)^{\frac{2}{p}}(\alpha_1s_1+1)(\alpha_2s_2+1)}\left[\left|\frac{\partial ^2f}{\partial \lambda \partial \mu}(a,c)\right|^q+m_2\alpha_2s_2\left|\frac{\partial ^2f}{\partial \lambda \partial \mu}\left(a, \frac{d}{m_2}\right)\right|^q\right.\\
&&\!\!\!\!\!\!\left.+m_1\alpha_1s_1\left|\frac{\partial ^2f}{\partial \lambda \partial \mu}\left(\frac{b}{m_1}, c\right)\right|^q+m_1m_2\alpha_1\alpha_2s_1s_2\left|\frac{\partial ^2f}{\partial \lambda \partial \mu}\left(\frac{b}{m_1}, \frac{d}{m_2}\right)\right|^q\right]^{\frac{1}{q}}.
\end{eqnarray}
where $p$ and $q$ are conjugate numbers.
\end{thm}
\begin{proof}
Applying H\"older Inequality on (\ref{t1a}), we have
\begin{eqnarray}\label{t2a}
&&\nonumber\!\!\!\!\!\!\!\!\!\! \left|\frac{f(a, c)+f(a,d)+f(b,c) + f(b,d)}{4} + \frac{1}{(b-a)(d-c)} \int_a^b \int_c^d f(x,y)dydx \!-\!\mathcal{A}\right|\\
&&\nonumber\indent\indent\indent\indent \leq \frac{(b-a)(d-c)}{4}\left(\int_0^1 \int_0^1 |(1-2\lambda)(1-2\mu)|^pd\lambda d\mu\right)^{\frac{1}{p}}\\
&&\indent\indent\indent\indent\indent\indent\left(\int_0^1\!\! \int_0^1\left|\frac{\partial ^2f}{\partial \lambda \partial \mu} \left(\lambda a\!+\!(1-\lambda)b, \mu c\!+\!(1-\mu)d\right)\right|^qd\lambda d\mu\right)^{\frac{1}{q}}
\end{eqnarray}
Here,
\begin{eqnarray}\label{t2b}
&&\nonumber\!\!\!\!\!\!\!\!\!\!\int_0^1 \int_0^1 |(1-2\lambda)(1-2\mu)|^p d\lambda d\mu = \int_0^1 \int_0^1 |(1-2\lambda)|^p |(1-2\mu)|^p d\lambda d\mu\\
&&\nonumber\!\!\!\!\!\!=\!\left(\!\! \int_0^{\frac{1}{2}}(1\!-\!2\lambda)^pd\lambda\!+\!\int_{\frac{1}{2}}^1(2\lambda\!-\!1)^pd\lambda\right)\left( \int_0^{\frac{1}{2}}(1\!-\!2\mu)^pd\mu\!+\!\int_{\frac{1}{2}}^1(2\mu\!-\!1)^pd\mu\right)\\
&&\!\!\!\!\!\!=\frac{1}{(p+1)^2}
\end{eqnarray}
Since $\left| \frac{\partial ^2f}{\partial t \partial s}\right|^q$ is $s-(\alpha,m)$ co-ordinated convex on $\Delta$, and so (\ref{t2a}) becomes
\begin{eqnarray}\label{t2c}
&&\nonumber\!\!\!\!\!\!\!\!\!\! \int_0^1\!\! \int_0^1\left|\frac{\partial ^2f}{\partial \lambda \partial \mu} \left(\lambda a\!+\!(1-\lambda)b, \mu c\!+\!(1-\mu)d\right)\right|^q d\lambda d\mu = \int_0^1\!\! \int_0^1\left[\lambda^{\alpha_1 s_1}\mu^{\alpha_2 s_2}\left|\frac{\partial ^2f}{\partial \lambda \partial \mu}(a, c)\right|^q\right.\\
&&\nonumber\indent\left.+\!m_2\!\lambda^{\alpha_1 s_1}\!(\!1\!-\!\mu^{\alpha_2 s_2}\!)\!\left|\frac{\partial ^2f}{\partial \lambda \partial \mu}\!\left(\!a, \frac{d}{m_2}\!\right)\right|^q\!+\!m_1\!\mu^{\alpha_2 s_2}\!(\!1\!-\!\lambda^{\alpha_1 s_1}\!)\left|\frac{\partial ^2f}{\partial \lambda \partial \mu}\left(\!\frac{b}{m_1}, c\!\right)\right|^q\right. \\
&&\nonumber\indent\indent\indent\indent\indent\indent\left. +m_1m_2(1-\lambda^{\alpha_1 s_1})(1-\mu^{\alpha_2 s_2})\left|\frac{\partial ^2f}{\partial \lambda \partial \mu}\left(\frac{b}{m_1}, \frac{d}{m_2}\right)\right|^q\right]d\lambda d\mu\\
&&\nonumber\!\!\!\!\!\!\!\!\!\!\!\!\!\!=\frac{1}{(\alpha_1s_1+1)(\alpha_2s_2+1)}\left(\left|\frac{\partial ^2f}{\partial \lambda \partial \mu}(a, c)\right|^q+m_2\alpha_2s_2\left|\frac{\partial ^2f}{\partial \lambda \partial \mu}\!\left(\!a, \frac{d}{m_2}\!\right)\right|^q\right.\\
&&\indent\indent\indent\left.+m_1\alpha_1s_1\left|\frac{\partial ^2f}{\partial \lambda \partial \mu}\left(\!\frac{b}{m_1}, c\!\right)\right|^q+m_1m_2\alpha_1\alpha_2s_1s_2\left|\frac{\partial ^2f}{\partial \lambda \partial \mu}\left(\frac{b}{m_1}, \frac{d}{m_2}\right)\right|^q\right)
\end{eqnarray}
(\ref{t2b}) and (\ref{t2c}) together implies (\ref{te2}).
\end{proof}
\begin{remark}
By adjusting $s_1=s_2=1$ and $(\alpha_1, m_1)=(\alpha_2, m_2) = (1,1)$ inequality $(\ref{te2})$ cuts down to \cite[Theorem 3]{r7}.
\end{remark}
\begin{thm}\label{T3}
 Let  a function $f : \Delta \subset \mathbb{R}^2 \rightarrow \mathbb{R}^2$ be a partially differentiable on $\Delta := [a, b] \times [c, d]$ in $\mathbb{R}^2$ with $a < b\,\,\wedge\,\,c < d$. If $\left| \frac{\partial ^2f}{\partial t \partial s}\right|^q$, for $q\geq1$ belonging to the class ${K_{m, 1}^{\alpha, s}}(\Delta)$ and $\frac{\partial ^2f}{\partial \lambda \partial \mu} \in L(\Delta)$, then the following inequality agrees:
 \begin{eqnarray}\label{te3}
&&\nonumber\!\!\!\!\!\!\!\!\!\!\!\!\left|\frac{f(a, c)+f(a,d)+f(b,c) + f(b,d)}{4} + \frac{1}{(b-a)(d-c)} \int_a^b \int_c^d f(x,y)dydx - \mathcal{A}\right|\\
&&\nonumber\!\!\!\!\!\! \leq \frac{(b-a)(d-c)}{4^{\frac{2q-1}{q}}}\left[\mathcal{BC}\left\{\left|\frac{\partial ^2f}{\partial \lambda \partial \mu}(a,c)\right|^q+m_1\left|\frac{\partial ^2f}{\partial \lambda \partial \mu}\left(\frac{b}{m_1}, c\right)\right|^q\right\}\right.\\
&&\!\!\!\!\!\!\left.+(1-\mathcal{B})(1-\mathcal{C})\left\{m_2\left|\frac{\partial ^2f}{\partial \lambda \partial \mu}\left(a, \frac{d}{m_2}\right)\right|^q+m_1m_2\left|\frac{\partial ^2f}{\partial \lambda \partial \mu}\left(\frac{b}{m_1}, \frac{d}{m_2}\right)\right|^q\right\}\right]^{\frac{1}{q}}.
\end{eqnarray}
where$p$ and $q$ are conjugate numbers and $\mathcal{B}$ and $\mathcal{C}$ comes from Theorem \ref{T1}.
\end{thm}
\begin{proof}
Applying power mean Inequality for double integral on (\ref{t1a}), we have
\begin{eqnarray}\label{t3a}
&&\nonumber\!\!\!\!\!\!\!\!\!\! \left|\frac{f(a, c)+f(a,d)+f(b,c) + f(b,d)}{4} + \frac{1}{(b-a)(d-c)} \int_a^b \int_c^d f(x,y)dydx \!-\!\mathcal{A}\right|\\
&&\nonumber\indent \leq \frac{(b-a)(d-c)}{4}\left(\int_0^1 \int_0^1 |(1-2\lambda)(1-2\mu)| d\lambda d\mu\right)^{1-\frac{1}{q}}\\
&&\indent\indent\left(\!\int_0^1\!\!\! \int_0^1\!\!|(1\!-\!2\lambda)(1\!-\!2\mu)|\!\left|\!\frac{\partial ^2f}{\partial \lambda \partial \mu} \left(\lambda a\!+\!(1\!-\!\lambda)b, \mu c\!+\!(1\!-\!\mu)d\right)\!\right|^q\!d\lambda\! d\mu\right)^{\frac{1}{q}}
\end{eqnarray}
Here,
\begin{eqnarray}\label{t3b}
&&\nonumber\!\!\!\!\!\!\!\!\!\!\int_0^1 \int_0^1 |(1-2\lambda)(1-2\mu)| d\lambda d\mu = \int_0^1 \int_0^1 |(1-2\lambda)| |(1-2\mu)| d\lambda d\mu\\
&&\nonumber\!\!\!\!\!\!=\!\left(\!\! \int_0^{\frac{1}{2}}(1\!-\!2\lambda)d\lambda\!+\!\int_{\frac{1}{2}}^1(2\lambda\!-\!1)d\lambda\right)\left( \int_0^{\frac{1}{2}}(1\!-\!2\mu)d\mu\!+\!\int_{\frac{1}{2}}^1(2\mu\!-\!1)d\mu\right)\\
&&\!\!\!\!\!\!=\frac{1}{4}
\end{eqnarray}
Since $\left| \frac{\partial ^2f}{\partial t \partial s}\right|^q$ is $s-(\alpha,m)$ co-ordinated convex on $\Delta$, and so (\ref{t3a}) becomes
\begin{eqnarray}\label{t3c}
&&\nonumber\!\!\!\!\!\!\!\!\!\! \int_0^1\!\! \int_0^1|(1-2\lambda)| |(1-2\mu)|\left|\frac{\partial ^2f}{\partial \lambda \partial \mu} \left(\lambda a\!+\!(1-\lambda)b, \mu c\!+\!(1-\mu)d\right)\right|^q d\lambda d\mu \\
 && \nonumber\indent\indent= \int_0^1\!\! \int_0^1|(1-2\lambda)| |(1-2\mu)|\left[\lambda^{\alpha_1 s_1}\mu^{\alpha_2 s_2}\left|\frac{\partial ^2f}{\partial \lambda \partial \mu}(a, c)\right|^q\right.\\
&&\nonumber\indent\left.+\!m_2\!\lambda^{\alpha_1 s_1}\!(\!1\!-\!\mu^{\alpha_2 s_2}\!)\!\left|\frac{\partial ^2f}{\partial \lambda \partial \mu}\!\left(\!a, \frac{d}{m_2}\!\right)\right|^q\!+\!m_1\!\mu^{\alpha_2 s_2}\!(\!1\!-\!\lambda^{\alpha_1 s_1}\!)\left|\frac{\partial ^2f}{\partial \lambda \partial \mu}\left(\!\frac{b}{m_1}, c\!\right)\right|^q\right. \\
&&\indent\indent\indent\indent\left. +m_1m_2(1-\lambda^{\alpha_1 s_1})(1-\mu^{\alpha_2 s_2})\left|\frac{\partial ^2f}{\partial \lambda \partial \mu}\left(\frac{b}{m_1}, \frac{d}{m_2}\right)\right|^q\right]d\lambda d\mu.
\end{eqnarray}
further, we have
\begin{eqnarray}\label{t3d}
&&\nonumber\!\!\!\!\!\!\!\!\!\!\!\!\!\!\!\! \int_0^1\!\!|(1-2\mu)|\!\!\left[\!\! \int_0^{\frac{1}{2}}\!\! |(1\!-\!2\lambda)|\! \left\{\lambda^{\alpha_1 s_1}\mu^{\alpha_2 s_2}\left|\frac{\partial ^2f}{\partial \lambda \partial \mu}(a, c)\right|^q\!+\!m_2\!\lambda^{\alpha_1 s_1}\!(\!1\!-\!\mu^{\alpha_2 s_2}\!)\!\left|\frac{\partial ^2f}{\partial \lambda \partial \mu}\!\left(\!a, \frac{d}{m_2}\!\right)\right|^q\right.\right.\\
&&\nonumber\!\!\!\!\!\!\!\!\!\!\!\!\!\!\!\!\!\!\left.\left.+\!m_1\!\mu^{\alpha_2 s_2}\!(\!1\!-\!\lambda^{\alpha_1 s_1}\!)\left|\!\frac{\partial ^2f}{\partial \lambda \partial \mu}\left(\!\frac{b}{m_1},\! c\!\right)\!\right|^q\!+\!m_1m_2(1\!-\!\lambda^{\alpha_1 s_1})(\!1\!-\!\mu^{\alpha_2 s_2}\!)\left|\!\frac{\partial ^2f}{\partial \lambda \partial \mu}\left(\!\frac{b}{m_1},\! \frac{d}{m_2}\!\right)\!\right|^q\!\right\}\!d\lambda\right. \\
&&\nonumber\!\!\!\!\!\!\!\!\!\!\!\!\!\!\!\!\!\!\!\!\left. +\!\!\int_{\frac{1}{2}}^1 |(1-2\lambda)| \left\{\lambda^{\alpha_1 s_1}\mu^{\alpha_2 s_2}\left|\frac{\partial ^2f}{\partial \lambda \partial \mu}(a, c)\right|^q+\!m_2\!\lambda^{\alpha_1 s_1}\!(\!1\!-\!\mu^{\alpha_2 s_2}\!)\!\left|\frac{\partial ^2f}{\partial \lambda \partial \mu}\!\left(\!a, \frac{d}{m_2}\!\right)\right|^q\right.\right.\\
&&\!\!\!\!\!\!\!\!\!\!\!\!\!\!\!\!\!\!\!\left.\left.\!+\!m_1\!\mu^{\alpha_2 s_2}\!(\!1\!\!-\!\!\lambda^{\alpha_1 s_1}\!)\left|\!\frac{\partial ^2f}{\partial \lambda \partial \mu}\!\left(\!\!\frac{b}{m_1}\!,\! c\!\!\right)\!\right|^q\!+\!m_1m_2(\!1\!-\!\lambda^{\alpha_1 s_1}\!)(\!1\!-\!\mu^{\alpha_2 s_2}\!)\left|\!\!\frac{\partial ^2f}{\partial \lambda \partial \mu}\!\left(\!\!\frac{b}{m_1}\!,\! \frac{d}{m_2}\!\!\right)\!\right|^q\!\right\}d\lambda\!\right]\!d\mu
\end{eqnarray}
Applying simple integration rules and simplifying, we have
\begin{eqnarray}\label{t3e}
&&\nonumber\!\!\!\!\!\!\!\!\!\!\!\!\!\!\!\! \int_0^1 |(1 -2\mu)|\left[\left\{\mu^{\alpha_2s_2}\left|\frac{\partial ^2f}{\partial \lambda \partial \mu}(a,c)\right|^q+m_2\left(1-\mu^{\alpha_2s_2}\right)\left|\frac{\partial ^2f}{\partial \lambda \partial \mu}\left(a, \frac{d}{m_2}\right)\right|^q\right\}\right.\\
&&\nonumber\indent\indent\indent\left.\left(\frac{1}{2^{\alpha_1s_1}(\alpha_1s_1+1)}-\frac{1}{2^{\alpha_1s_1}(\alpha_1s_1+2)}+\frac{2}{\alpha_1s_1+2}-\frac{1}{\alpha_1s_1+1}\right)\right.\\
&&\nonumber\indent\left.+\left\{m_1\mu^{\alpha_2s_2}\left|\frac{\partial ^2f}{\partial \lambda \partial \mu}\left(\frac{b}{m_1},c\right)\right|^q+m_1m_2\left(1-\mu^{\alpha_2s_2}\right)\left|\frac{\partial ^2f}{\partial \lambda \partial \mu}\left(\frac{b}{m_1}, \frac{d}{m_2}\right)\right|^q\right\}\right.\\
&&\indent\indent\left.\left(\!\frac{1}{2}\!-\!\frac{1}{2^{\alpha_1s_1}(\alpha_1s_1\!+\!1)}-\frac{1}{2^{\alpha_1s_1}(\alpha_1s_1\!+\!2)}+\frac{2}{\alpha_1s_1\!+\!2}-\frac{1}{\alpha_1s_1\!+\!1}\!\right)\!\right]d\mu
\end{eqnarray}
Now further, we have
\begin{eqnarray}\label{t3f}
&&\nonumber\!\!\!\!\!\!\!\!\!\!\!\!\!\!\!\! \int_0^{\frac{1}{2}} |(1 -2\mu)|\left[\left\{\mu^{\alpha_2s_2}\left|\frac{\partial ^2f}{\partial \lambda \partial \mu}(a,c)\right|^q+m_2\left(1-\mu^{\alpha_2s_2}\right)\left|\frac{\partial ^2f}{\partial \lambda \partial \mu}\left(a, \frac{d}{m_2}\right)\right|^q\right\}\left(\mathcal{B}\right)\right.\\
&&\nonumber\!\!\!\!\!\!\!\!\!\!\!\!\!\!\!\!\!\!\left.+\!\left\{\!m_1\mu^{\alpha_2s_2}\left|\frac{\partial ^2f}{\partial \lambda \partial \mu}\left(\frac{b}{m_1}\!,\!c\right)\right|^q\!\!+\!\!m_1m_2\left(1\!-\!\mu^{\alpha_2s_2}\right)\left|\frac{\partial ^2f}{\partial \lambda \partial \mu}\left(\frac{b}{m_1}\!,\! \frac{d}{m_2}\right)\right|^q\right\}\left(\!\frac{1}{2}\!-\!\mathcal{B}\!\right)\!\right]d\mu\\
&&\nonumber\!\!\!\!\!\!\!\!\!\!\!\!\!\!\!\!+\int_{\frac{1}{2}}^1 |(2\mu-1)|\left[\left\{\mu^{\alpha_2s_2}\left|\frac{\partial ^2f}{\partial \lambda \partial \mu}(a,c)\right|^q+m_2\left(1-\mu^{\alpha_2s_2}\right)\left|\frac{\partial ^2f}{\partial \lambda \partial \mu}\left(a, \frac{d}{m_2}\right)\right|^q\right\}\left(\mathcal{B}\right)\right.\\
&&\!\!\!\!\!\!\!\!\!\!\!\!\!\!\!\!\!\!\left.+\!\left\{\!m_1\mu^{\alpha_2s_2}\left|\frac{\partial ^2f}{\partial \lambda \partial \mu}\!\left(\!\frac{b}{m_1}\!,\!c\right)\!\right|^q\!\!\!+\!\!m_1m_2\left(1\!-\!\mu^{\alpha_2s_2}\right)\left|\frac{\partial ^2f}{\partial \lambda \partial \mu}\!\left(\!\frac{b}{m_1}\!,\! \frac{d}{m_2}\!\right)\right|^q\!\right\}\!\left(\!\frac{1}{2}\!-\!\mathcal{B}\!\right)\!\right]\!d\mu
\end{eqnarray}
Where $\mathcal{B}$ is already defined in Theorem \ref{T1} as
$$\mathcal{B}=\frac{1}{2^{\alpha_1s_1}(\alpha_1s_1+1)}-\frac{1}{2^{\alpha_1s_1}(\alpha_1s_1+2)}+\frac{2}{\alpha_1s_1+2}-\frac{1}{\alpha_1s_1+1}$$
After executing simple integration on (\ref{t3f}) and appointing $$\mathcal{C}=\frac{1}{2^{\alpha_2s_2}(\alpha_2s_2+1)}-\frac{1}{2^{\alpha_2s_2}(\alpha_2s_2+2)}+\frac{2}{\alpha_2s_2+2}-\frac{1}{\alpha_2s_2+1}$$
we have (\ref{te3}). Which fills out the proof.
\end{proof}
\begin{remark}
By adjusting $s_1=s_2=1$ and $(\alpha_1, m_1)=(\alpha_2, m_2) = (1,1)$, $(\ref{te3})$ cuts down to \cite[Theorem 4]{r7}.
\end{remark}

\end{document}